\documentclass[letterpaper,11pt]{article}

\usepackage[square,numbers]{natbib}

\usepackage[utf8]{inputenc}   
\usepackage[T1]{fontenc}  

\usepackage{pgf,tikz}
\usetikzlibrary{arrows}
\usepackage{multicol}
\usepackage{amsthm}
\usepackage{amsmath}
\usepackage{amssymb}
\usepackage{amsfonts}
\usepackage{stmaryrd}
\usepackage{latexsym}
\usepackage{color}
\usepackage{graphics,graphicx,graphpap}
\usepackage{multirow}
\usepackage{rotating}
\usepackage[new]{old-arrows}
\usepackage[all]{xy}
\usepackage[letterpaper]{geometry}
\usepackage{subfigure}
\usepackage{hyperref}

\theoremstyle{plain}

\usepackage{setspace}
\usepackage{tocbibind}

\DeclareMathAlphabet{\mathpzc}{OT1}{pzc}{m}{it}
\newtheorem{theorem}{Theorem}

\newtheorem{que}{Question}

\title{Homotopy type through homology groups}
\author{Omar Antolín Camarena, Andrés Carnero Bravo}
\date{}

\begin{document}
\maketitle
The homology groups of a space do not determine its homotopy type, but when the space is a simply-connected CW-complex, it is a well-known fact that if the space has the homology of a wedge of spheres of the same dimension, then it must have the homotopy type of that wedge 
of spheres. So we can ask if there are other cases where the homology groups of a simply connected CW-complex determine 
that the space is a wedge of spheres. In 
this brief note we show that if a complex has free finitely generated reduced homology groups for two consecutive dimensions and trivial homology 
for all other dimensions, then it must have the homotopy type of a wedge of spheres of two consecutive dimensions. We also show  
other pairs of dimensions for which the last result can be generalized. 

To fix ideas, we first give a proof of the above mentioned folklore result; the map constructed in this proof will be use in the proof of our main theorem. Throughout we will take homology with integral coefficients.
\begin{theorem}
Let $X$ be a simply connected CW-complex such that the only non-zero reduced homology group is $\tilde{H}_d(X)\cong\mathbb{Z}^a$. Then 
\[X\simeq\bigvee_a\mathbb{S}^d\]
\end{theorem}
\begin{proof}
By the Hurewicz Theorem, $\pi_d(X)\cong\tilde{H}_d(X)\cong\mathbb{Z}^a$. Therefore, there are maps 
$i_j\colon\mathbb{S}^d\longrightarrow X$
for $1\leq j\leq a$ such that the combined map
\[i\colon\bigvee_a\mathbb{S}^d\longrightarrow X\]
is an isomorphism on $\pi_d$. Thus $i$ induces an isomorphism on reduced homology groups and, by Whitehead's Theorem, is a
homotopy equivalence.
\end{proof}

The case $k=1$ of the following theorem is a special case of example 4C.2 of \citep{hatcher}. For completness we give a different proof.
\begin{theorem}
Let $X$ be a simply-connected CW-complex such that
\[\tilde{H}_q(X)\cong\left\lbrace\begin{array}{cc}
   \mathbb{Z}^a  & \mbox{ for } q=d \\
    \mathbb{Z}^b & \mbox{ for } q=d+k \\
    0 & \mbox{ for } q\neq d,d+k
\end{array}\right.\]
where $a,b,d,k$ are positive integers, with $d>k$ and
$k \in \{1,5,6,13,62\}$. Then 
\[X\simeq\bigvee_{a}\mathbb{S}^d\vee\bigvee_{b}\mathbb{S}^{d+k}\]
\end{theorem}
\begin{proof}
The key step of the proof is to show that $\pi_{d+k}(X)\cong\mathbb{Z}^b\oplus H$ for some $H$ so we can pick a map 
\[s\colon\bigvee_b\mathbb{S}^{d+k}\longrightarrow X\]
with $\mathrm{im}(s_*)=\mathbb{Z}^b \le \pi_{d+k}(X)$.

First, we do this for the case $k=1$ using a Serre spectral 
sequence argument. The $d$-th space $P_d(X)$ in the Postnikov tower is 
given by $K(\mathbb{Z}^a, d)$. Let $\phi_d$ be the canonical map 
$\phi_d\colon X\longrightarrow P_d(X)$. We have that $\phi_d$ is
$(d+1)$-connected and that $F:=\mathsf{hofib}(\phi_d)$ is $d$-connected. 
In the Serre spectral sequence for $F\longrightarrow X\longrightarrow K(Z^a,d)$, we have that
\[E_{p,q}^2=H_p(K(\mathbb{Z}^a,d),H_q(F))\]
Thus $E_{p,q}^2 = 0$ for any $1\leq q\leq d$ and for any $1\leq p<d$. 
Also $E_{0,d+1}^2=H_{d+1}(F)\cong\pi_{d+1}(F)\cong\pi_{d+1}(X)$
and, $E_{d+1,0}^2 = H_{d+1}(K(\mathbb{Z}^a, d)) = 0$ (the easiest way
to see this last fact is that one can build a $K(\mathbb{Z}^a, d)$ by
killing the homotopy groups of $\bigvee_a S^d$ in degrees $n+1$ and
higher and this is done by attaching cells of dimensions $n+2$ and higher).

Therefore, the second page of the sequence looks like:
\begin{equation*}
    \xymatrix{
    \pi_{d+1}(X) & E_{1,d+1}^2 & E_{2,d+1}^2 & \cdots & E_{d,d+1}^2 & E_{d+1,d+1}^2 & E_{d+2,d+1}^2\\
    0  & 0  & 0  \ar@{->}[ull]& \cdots &   0 & 0& 0\ar@{->}[ull]\\
    \vdots  & \vdots  & \vdots &  &   \vdots & \vdots& \vdots\\
    0  & 0  & 0  & \cdots &   0 & 0& 0\\
    0  & 0  & 0  \ar@{->}[ull]& \cdots &   0 & 0& 0\ar@{->}[ull]\\
    \mathbb{Z} & 0 & 0 \ar@{->}[ull]& \cdots& \mathbb{Z}^a & 0 & E_{d+2,0}^2 \ar@{->}[ull]
    }
\end{equation*}
Then $\mathbb{Z}^b\cong \tilde{H}_{d+1}(X)\cong E_{0,d+1}^\infty\cong\pi_{d+1}(X)/H$, where 
$H=im\left(d_{d+2}^{^{d+2,0}}\right)$. The short exact sequence $0 \to H 
\to \pi_{d+1}(X) \to \mathbb{Z}^a \to 0$ splits because $\mathbb{Z}^a$ is
free abelian, and so 
\[\pi_{d+1}(X)\cong\mathbb{Z}^b\oplus H.\]

For the other cases, namely $k \in \{5,6,13,62\}$, consider the following 
diagram:
\begin{equation*}
    \xymatrix{
    Q \ar@/^0.6pc/@{->}[rrd]^{g}
    \ar@/_0.6pc/@{->}[rdd] \ar@{->}[rd]_{f}& & \\
     & \displaystyle\bigvee_a\mathbb{S}^d \ar[r]_{i} \ar@{->}[d] & X \ar[d]^{q}\\
     & \ast \ar@{->}[r] & C,}
\end{equation*}
where the inner square is an homotopy push-out and the outer square is
a homotopy pullback, i.e., $Q = \mathsf{hofib}(q)$ and $C = 
\mathsf{hocofib}(i)$. Then,  by the Blakers-Massey Theorem, $f$ is
$(2d-1)$-connected. Next we will look at the long exact sequence of 
homotopy groups for the fiber sequence $Q \longrightarrow X 
\longrightarrow C$ to compute homotopy groups of $X$ up to dimension 
$d+k$. For this we will need some information about the homotopy groups 
of $Q$ and $C$.

By the van Kampen theorem, $C$ is simply-connected, and the Mayer-
Vietoris homology sequence for $C$, tells us that $C$ has reduced 
homology $\mathbb{Z}^b$ in degree $d+k$ and zero otherwise, so that $C 
\simeq \bigvee_b \mathbb{S}^{d+k}$. The long exact sequence in
homotopy then gives us $\pi_r(Q) \cong \pi_r(X)$ for $r \le d+k-2$.

As for $Q$, the Blakers-Massey theorem gave us that $\pi_r(Q) \cong 
\pi_r(\bigvee_a \mathbb{S}^d)$ for $r<2d-1$. Taking $
\bigvee_a\mathbb{S}^d$ as the $d$-skeleton of 
$\prod_a\mathbb{S}^d$, we have that the pair 
$\left(\prod_a\mathbb{S}^d,\bigvee_a\mathbb{S}^d\right)$ is 
$(2d-1)$-connected, therefore 
\[ \pi_r(Q) \cong \pi_r\left(\bigvee_a\mathbb{S}^d\right) \cong \bigoplus_a\pi_r(\mathbb{S}^d)\]
for all $r<2d-1$. Since $d>k$ we have $d+k-1<2d-1$ and therefore,
\[\pi_{d+k-1}(Q)\cong \bigoplus_a\pi_{d+k-1}(\mathbb{S}^d).\]
These groups are zero for the pairs $d$ and $k>1$ given in the theorem 
(see \citep{Toda_1963,Wang_2017}), and are in fact exactly the stable homotopy 
groups of spheres that are known to be zero.

Now we see that the long exact sequence in homotopy near degree $d+k$ is:
\[\pi_{d+k}(Q) \xrightarrow{g_\ast} \pi_{d+k}(X) \xrightarrow{q_\ast}
 \pi_{d+k}(C) \to 0 \to \pi_{d+k-1}(X) \to 0.\]
Therefore $\pi_{d+k-1}(X) = 0$ and, since $\pi_{d+k}(C)$ is free abelian, we get the desired splitting
\[\pi_{d+k}(X)\cong \mathbb{Z}^b\oplus\mathrm{im}(g_*).\]

Now in all cases we can pick a map 
\[s\colon\bigvee_b\mathbb{S}^{d+k}\longrightarrow X\]
with $\mathrm{im}(s_*)=\mathbb{Z}^b \le \pi_{d+k}(X)$, and we will show 
that the map 
\[i\vee s\colon\bigvee_a\mathbb{S}^d\vee\bigvee_b\mathbb{S}^{d+k}\longrightarrow X\]
is a homotopy equivalence. Since its source and target are simply-connected,
it is enough to show it induces and isomorphism on homology, 
which it clearly does in degree $d$.

All that is left to show is that $s$ is an isomorphism on $H_{d+k}$. In 
the case $k = 1$ this is because we had an isomorphism 
$\pi_{d+1}(X)/H \cong H_{d+1}(X)$ coming from the Serre spectral 
sequence. In the other cases, where $k>1$, the summand $\mathbb{Z}^b$ of 
$\pi_{d+k}(X)$ was $\pi_{d+k}(C)$, so that $q_\ast \circ s_\ast : 
\pi_{d+k}(\bigvee_b \mathbb{S}^{d+k}) \to \pi_{d+k}(C)$ is an 
isomorphism. By the Hurewicz theorem, $q \circ s$ is an isomorphism on 
$H_{d+k}$, and since $q$ is also an isomorphism on $H_{d+k}$ we conclude 
that $s$ is as well.
\end{proof}

The lower bound $d>k$ in the theorem can not be reduced. The product
$X=\mathbb{S}^k\times\mathbb{S}^k$ is not homotopy equivalent to
a wedge of spheres and is a counterexample to a hypothetical $d \ge k$ 
generalization of the theorem. It should be noticed that for any
CW-complex $X$ which satisfies the hypothesis of the theorem but with
$d=k$ instead of $d>k$, the suspension $\Sigma X$ does satisfy all the 
hypothesis and thus has the homotopy type of a wedge of spheres. 

\begin{que}
Are there more pairs $q$ and $k$ such that the result holds for dimensions $d$ and $d+k$ for all $d\geq q$?
\end{que}

Certainly if any other stable stem $\pi^s_m$ were $0$, then $(q,k)=
(m+2,m+1)$ would be an example of such a pair, but notice the theorem
also holds for $(q,k) = (2,1)$ even though $\pi^s_0 \cong \mathbb{Z}$ is 
not zero.

Also note that for $k=2$ there is no such $q$: consider the space
$\Sigma^n\mathbb{CP}^2$ which is simply-connected, only has nonzero
reduced homology in degrees $n$ and $n+2$ (where it is $\mathbb{Z}$),
but is not homotopy equivalent to a wedge of spheres ---because 
$\mathrm{Sq}^2$ is nonzero, for example. Similarly, suspensions of
the quaternionic projective plane show that for $k=4$ there can be
no $q$. 
The $\Sigma^n\mathbb{CP}^2$ example also shows that the theorem
can not be generalized to spaces with free abelian homology in three
or more consecutive dimensions.

\textbf{Acknowledgments.} We wish to thank Allen Hatcher for pointing out that the case $k=1$ of the main theorem is a special case of Proposition 4C.1 in his textbook \citep{hatcher}. We also wish to thank Oliver Röndigs for pointing out we could add the case $k=62$ to the main theorem.

\bibliographystyle{acm}
\bibliography{Homttrhoml}

\begin{thebibliography}{1}

\bibitem{hatcher}
{\sc Hatcher, A.}
\newblock {\em Algebraic Topology}.
\newblock Cambridge, 2002.

\bibitem{Toda_1963}
{\sc Toda, H.}
\newblock {\em Compositional Methods in Homotopy Groups of Spheres. ({AM}-49)}.
\newblock Princeton University Press, dec 1963.

\bibitem{Wang_2017}
{\sc Wang, G., and Xu, Z.}
\newblock The triviality of the 61-stem in the stable homotopy groups of
  spheres.
\newblock {\em Annals of Mathematics 186}, 2 (sep 2017).

\end{thebibliography}
\vspace{1cm}
\hspace{0.5cm}Omar Antolín Camarena

Instituto de Matemáticas, UNAM, Mexico City, Mexico

\textit{E-mail address:} \href{mailto:omar@matem.unam.mx}{omar@matem.unam.mx}

\vspace{1cm}
Andrés Carnero Bravo

Instituto de Matemáticas, UNAM, Mexico City, Mexico

\textit{E-mail address:} \href{mailto:acarnerobravo@gmail.com}{acarnerobravo@gmail.com}
\end{document}